\documentclass{amsart}
\usepackage{amscd,amssymb,amsopn,amsmath,amsthm,graphics,amsfonts,accents,enumerate,verbatim,calc}
\usepackage[dvips]{graphicx}
\usepackage[colorlinks=true,linkcolor=red,citecolor=blue]{hyperref}
\usepackage[all]{xy}

\usepackage{tikz}

\calclayout

\newcommand{\rt}{\rightarrow}
\newcommand{\lrt}{\longrightarrow}

\newcommand{\va}{\vartheta}
\newcommand{\st}{\stackrel}
\newcommand{\pa}{\partial}

\newcommand{\la}{\lambda}
\newcommand{\La}{\Lambda}

\newcommand{\C}{\mathbb{C} }
\newcommand{\D}{\mathbb{D} }
\newcommand{\K}{\mathbb{K} }

\newcommand{\Z}{\mathbb{Z} }

\newcommand{\CA}{\mathcal{A} }
\newcommand{\CC}{\mathcal{C} }

\newcommand{\CK}{\mathcal{K} }

\newcommand{\CX}{\mathcal{X} }

\newcommand{\BZ}{\mathbf{Z}}
\newcommand{\CB}{\mathcal{B} }

\newcommand{\X}{\mathbf{X}}
\newcommand{\Y}{\mathbf{Y}}

\newcommand{\T}{{\mathfrak{T}}}

\newcommand{\Mod}{{\rm{Mod\mbox{-}}}}

\newcommand{\perf}{{\rm{perf}}}

\newcommand{\mmod}{{\rm{{mod\mbox{-}}}}}
\newcommand{\mmodd}{{\rm{mod}}_0\mbox{-}}

\newcommand{\Inj}{{\rm{Inj}\mbox{-}}}
\newcommand{\Prj}{{\rm{Prj}\mbox{-}}}
\newcommand{\prj}{{\rm{prj}\mbox{-}}}

\newcommand{\inj}{{\rm{inj}\mbox{-}}}

\newcommand{\im}{{\rm{Im}}}
\newcommand{\op}{{\rm{op}}}

\newcommand{\add}{{\rm{add}\mbox{-}}}
\newcommand{\coh}{{\rm{coh}}}

\newcommand{\ac}{{\rm{ac}}}

\newcommand{\id}{{\rm{id}}}

\newcommand{\bb} {{\rm{b}}}

\newcommand{\Ker}{{\rm{Ker}}}

\newcommand{\Hom}{{\rm{Hom}}}
\newcommand{\Ext}{{\rm{Ext}}}
\newcommand{\End}{{\rm{End}}}

\theoremstyle{plain}
\newtheorem{theorem}{Theorem}[section]
\newtheorem{corollary}[theorem]{Corollary}
\newtheorem{lemma}[theorem]{Lemma}

\newtheorem{proposition}[theorem]{Proposition}

\theoremstyle{definition}
\newtheorem{definition}[theorem]{Definition}

\newtheorem{remark}[theorem]{Remark}

\theoremstyle{plain}

\theoremstyle{definition}

\numberwithin{equation}{section}

\newtheorem{sstheorem}{Theorem}

\begin{document}

\title[Categorical Resolutions]{A Functorial Approach to Categorical Resolutions}

\author[R. Hafezi and M. H. Keshavarz]{R. Hafezi and M. H. Keshavarz}

\address{School of Mathematics, Institute for Research in Fundamental Sciences (IPM), P.O.Box: 19395-5746, Tehran, Iran}
\email{hafezi@ipm.ir}
\email{keshavarz@ipm.ir}

\subjclass[2010]{18E30, 16E35, 16G10, 14E15}

\keywords{Categorical resolutions, Artin algebras, Auslander's formula}


\begin{abstract}
Using a relative version of Auslander's formula, we give a functorial approach to show that the bounded derived category of every Artin algebra admits a categorical resolution. This, in particular, implies that the bounded derived categories of Artin algebras of finite global dimension determine bounded derived categories of all Artin algebras. Hence, this paper can be considered as a typical application of functor categories, introduced in representation theory by Auslander, to categorical resolutions.
\end{abstract}

\maketitle


\section{Introduction}\label{1}
Let $X$ be a singular algebraic variety of finite type over an algebraically closed field. A resolution of singularities of $X$ is a certain (proper and birational) morphism $\pi:\widetilde{X} \rt X$, where $\widetilde{X}$ is a non-singular algebraic variety. As remarked in \cite[Page 1, Line -4]{Kuz}, the functor $\D^{\bb}(\coh\widetilde{X}) \rt \D^{\bb}(\coh X)$ induced by $\pi$ enjoys some remarkable properties. 
The bounded derived categories of coherent sheaves on $\widetilde{X}$ and $X$ are related by two natural functors, known as the derived pushforward
$\pi_{\ast}:\D^{\bb}(\coh \widetilde{X}) \rt \D^{\bb}(\coh X)$ and the derived pullback functor $\pi^{\ast}:\D^{\perf}(\coh X) \rt \D^b(\coh \widetilde{X})$, such that $\pi^{\ast}$ is left adjoint to $\pi_{\ast}$. Here $ \D^{\bb}(\coh X) $ stands for the bounded derived categories of coherent sheaves on $X$ and $\D^{\perf}(\coh X)$ stands for the full subcategory of $\D^{\bb}(\coh X)$ consisting of perfect complexes. If furthermore, $X$ have rational singularities, then the composition $ \pi_{\ast} \circ \pi^{\ast} $ is isomorphic to the identity functor \cite[Page 2]{Kuz}.

Based on this observation, as he mentioned, Kuznetsov \cite[Page 2]{Kuz} introduced the notion of a categorical resolution of singularities. By definition a categorical resolution of a triangulated category $\D$ is a regular triangulated category $\widetilde{\D}$ and a pair of functors $\pi_{\ast}$ and $\pi^{\ast}$ satisfying almost similar conditions as above, see \cite[Definition 3.2]{Kuz}.
Recall that a triangulated category $\T$ is called regular if it is equivalent to an admissible subcategory of the bounded derived category of a smooth algebraic variety \cite[Definition 3.1]{Kuz}.

Note that, as remarked in  \cite[Page 7, Line -1]{BO},  the pushforward functor $\pi_{\ast}$ identifies $\D^{\bb}(\coh X)$ with the quotient of $\D^{\bb}(\coh \widetilde{X})$ by the kernel of $\pi_{\ast}$. Based on this observation, as they mentioned, Bondal and Orlov \cite{BO} defined a categorical desingularization of a triangulated category $\T$ as a pair $(\CA,\CK)$, where $\CA$ is an abelian category of finite homological dimension and $\CK$ is a thick triangulated subcategory of $\D^{\bb}(\CA)$ such that $\T \simeq \D^{\bb}(\CA)/\CK$, see \cite[\S 5]{BO}. 
Recall that a full triangulated subcategory of a triangulated category $\T$ is called thick if it is closed under taking direct summands.

Recently, Zhang \cite{Z} combined these two categorical levels of the notion of a resolution of singularities and suggested a new definition for a categorical resolution of a non-smooth triangulated category \cite[Definition 2.2]{Z}. 

He then proved that if $\La$ is an Artin algebra of infinite global dimension and has a module $T$ with $\id_{\La}T < 1$ such that $ {}^{\perp}T $ is of finite type, then the bounded derived category $\D^b(\mmod\Lambda)$ admits a categorical resolution \cite[Theorem 6.1]{Z}. The main technique for proving this result is the notion of relative derived categories studied by several authors in different settings, see e.g \cite{N}, \cite{Bu} and \cite{GZ}. Recall that a triangulated category $\T$ is called smooth if it is triangle equivalent to the bounded derived category of an abelian category $\CA$ with $ \D^b (\CA)= \D^b_{hf}(\CA)$, where $\D^b_{hf}(\CA)$ is the subcategory of $ \D^b (\CA)$ consisting of homological finite objects \cite[Definition 2.1]{Z}.

In this paper, we generalize Zhang's result to arbitrary Artin algebras and show that bounded derived category of every Artin algebra admits a categorical resolution. The technique for the proof is based on a relative version of the so-called Auslander's Formula \cite{Au1} and \cite{L}.  Besides Auslander's formula we use the following known result of Auslander. In his Queen Mary College lectures \cite[\S III.3]{Au2} he proved the following important result:

\begin{sstheorem}
Let $ \Lambda $ be an Artin algebra with radical $ J $ and $ n $ be the  nilpotency index of $ J $. Then the Artin algebra $\widetilde{{\La}} =\End_{\Lambda} (\bigoplus_{i=1}^{n}\Lambda/J^i)^{\op} $ has the following properties:
\begin{itemize}
\item[$ (i) $] The global dimension of $ \widetilde{{\La}} $ is finite;
\item[$ (ii) $] There is a finitely generated projective $ \widetilde{{\La}} $-module $ P $ such that $ \End_{\widetilde{{\La}}} (P)^{\op} $ and $ \widetilde{{\La}} $  are isomorphic Artin algebras.
\end{itemize}
\end{sstheorem}

Hence, as he mentioned, Artin algebras of finite global dimension determine all Artin algebras { \cite[Page 47, Line 34]{Au2}}. Our volunteer for the proof of the main theorem via a functorial approach is $\widetilde{{\La}}$, that throughout for ease of reference we call it A-algebra of $\La$. `A' stands both for `Auslander' and also `Associated' algebra. Of course, instead of Auslander's result, one can also use of an important result due to O. Iyama \cite[Corollary 1.2]{I}.
Note that, as remarked in \cite[Page 1]{DR}, there exists an idempotent $e$ of $\widetilde{{\La}}$ such that $\La \simeq e\widetilde{{\La}}e$.
For the proof see \cite[Page 243, Line -2]{Wa}.


\section{Preliminaries}\label{Preliminaries}

Throughout the paper, $\Lambda$ denotes an Artin algebra over a commutative Artinian ring $R$, $\Mod\Lambda$ denotes the category of all right $\Lambda$-modules and $\mmod\Lambda$ denotes its full subcategory consisting of all finitely presented modules. Moreover, $\Prj\La$ [resp. $\prj\La$] denotes the full subcategory of $\Mod\La$ [resp. $\mmod\La$] consisting of projective [resp. finitely generated projective] modules. Similarly, the subcategories $\Inj\La$ and $\inj\La$ are defined. For a $\La$-module $M$, we let $\add M$ denote the class of all modules that are isomorphic to a direct summand of a finite direct sum of copies of $M$.

\subsection{Functor Category}\label{Functor Category}
Let $\CA$ be an additive skeletally small category. The Hom sets will be denoted either by $\Hom_{\CA}( - , - )$, $\CA( - , - )$ or even just $( - , - )$, if there is no risk of ambiguity. Let $\CX$ be a full subcategory of $\CA$. By definition, a (right) $\CX$-module is a contravariant additive functor $F:\CX \rt \CA b$, where $\CA b$ denotes the category of abelian groups. The $\CX$-modules and natural transformations between them, called morphisms, form an abelian category denoted by $\Mod\CX$ or sometimes $(\CX^{\op},\CA b)$. 
\vspace{2 cm}

An $\CX$-module $F$ is called finitely presented if there exists an exact sequence $$\CX( - ,X) \rt \CX( - ,X') \rt F \rt 0,$$ with $X$ and $X'$ in $\CX$. All finitely presented $\CX$-modules form a full subcategory of $\Mod \CX$, denoted by $\mmod\CX$ or sometimes $\rm{f.p.}(\CC^{op}, \CA b)$. 
Covariant additive functors and its full subcategory consisting of finitely presented (left) $\CX$-modules will be denoted by $\CX$-Mod and $\CX$-${\rm mod}$, respectively. Since every finitely generated subobject of a finitely presented object is finitely presented { \cite[Page 200, Line 7]{Au1}}, Auslander called them coherent functors. { Recall that an object $ A $ in $ \CA $ is finitely generated if given any family $ \lbrace A_i \rbrace_{i \in I} $ and an epimorphism $ f_i:A_i \longrightarrow A $ (with $ i \in I $), there exists a finite subset $ J $ of $ I $ such that $ f_j:A_j \longrightarrow A $ (with $ j \in J $) is an epimorphism \cite[Page 192, Line 6]{Au1}.}

The Yoneda embedding $\CX \hookrightarrow \mmod\CX$, sending each $X \in \CX$ to $\CX( - ,X):=\CA( - ,X)\vert_{\CX}$, is a fully faithful functor. Note that for each $X \in \CX$, $\CX( - ,X)$ is a projective object of $\mmod\CX$. 
Let $A \in \CA$. A morphism $\varphi : X \rt A$ with $X \in \CX$ is called a right $\CX$-approximation of $A$ if $\CA( - , X)\vert_{\CX} \lrt \CA( - ,A)\vert_{\CX} \lrt 0$ is exact, where
$\CA( - ,A) \vert_{\CX}$ is the functor $\CA( - ,A)$ restricted to $\CX$. Hence $A$ has a right $\CX$-approximation if and only if $( - ,A)\vert_{\CX}$ is a finitely generated object of $\Mod\CX$. $\CX$ is called contravariantly finite if every object of $\CA$ admits a right $\CX$-approximation. Dually, one can define the notion of left $\CX$-approximations and covariantly finite subcategories. $\CX$ is called functorially finite, if it is both covariantly and contravariantly finite.

\subsection{Relative Auslander's Formula}\label{RelAusFormula}
Let $\CA$ be an abelian category. Auslander's work on coherent functors { \cite[page 205, Line 8]{Au1}} implies that the Yoneda functor $\CA \lrt \mmod\CA$ induces a localisation sequence of abelian categories
\[\xymatrix{ \mmodd\CA \ar[rr]  && \mmod\CA  \ar[rr]^{} \ar@/^1pc/[ll] && \CA  \ar@/^1pc/[ll]^{} }\]
where $\mmodd\CA$ is the full subcategory of $\mmod\CA$ consisting of those functors $F$ for them there exists a presentation $\CA( - ,A) \lrt \CA( - ,A') \lrt F \lrt 0$ such that $A \lrt A'$ is an epimorphism, { see for more details} \cite[Theorem 2.2]{Kr1}, where $\mmodd\CA$ is denoted by ${\rm eff}\CA$.
This, in particular, implies that the functor $\mmod\CA \lrt \CA$, that is the left adjoint of the Yoneda functor, induces an equivalence $$\frac{\mmod\CA}{\mmodd\CA}\simeq \CA.$$ Following Lenzing \cite{L} this equivalence will be called the Auslander's formula.

In \cite[Theorem 3.7]{AHK},  the authors showed that for  every contravariantly finite subcategory $\CX$ of $\mmod \Lambda$ containing projective $\Lambda$-modules, there exists a recollement of abelian categories
\[\xymatrix{\mmodd \CX \ar[rr]^{i}  && \mmod \CX \ar[rr]^{\va} \ar@/^1pc/[ll]^{i_{\rho}} \ar@/_1pc/[ll]_{i_{\la}} && \mmod \Lambda \ar@/^1pc/[ll]^{\va_{\rho}} \ar@/_1pc/[ll]_{\va_{\la}} }\]
where $\mmodd\CX:=\Ker(\va)$, the full subcategory of $\mmod\CX$ consisting of all functors $F$ such that $\va(F)=0$. This, in particular, implies that $$\frac{\mmod\CX}{\mmodd\CX}\simeq \mmod \Lambda.$$
{The case} $\CX=\mmod \Lambda$,  gives the usual Auslander's formula.

{ Recall that a recollement of abelian category $\CA$ with respect to abelian categories $\CA'$ and $\CA''$ is a diagram
\[\xymatrix{\CA'\ar[rr]^{u}  && \CA \ar[rr]^{v} \ar@/^1pc/[ll]^{u_{\rho}} \ar@/_1pc/[ll]_{u_{\la}} && \CA'' \ar@/^1pc/[ll]^{v_{\rho}} \ar@/_1pc/[ll]_{v_{\la}} }\]
of additive functors such that $u$, $v_{\la}$ and $v_{\rho}$ are fully faithful, $(u_{\la},u)$, $(u,u_{\rho})$, $(v_{\la},v)$ and $(v,v_{\rho})$ are adjoint pairs and $\im (u)= \Ker (v)$ \cite[Definition 2.7]{PV}.

Also a localisation sequence consists only the lower  two rows of a recollement such that the functors appearing in them satisfy all the conditions of a recollement that involve only these functors.}

The following remark is devoted to recall the definitions of the functors $ \va \ and \ \va_{\rho} $.

\begin{remark}\label{Definition of Va}
$ (i) $
Let $ \CX $ be a contravariantly finite subcategory of $\mmod \Lambda$ containing projective $\Lambda$-modules,
$F \in \mmod \CX$, and $\CX( - ,X_1) \st{(-,d)} \lrt \CX( - ,X_0) \lrt F \lrt 0$ be a projective presentation of $F$, where $X_0$ and $X_1$ are in $\CX$. Then by \cite[Remark 3.2]{AHK}, $\va(F)$ is determined by the exact sequence
$$X_1 \st{d}{\lrt} X_0 \lrt \va(F) \lrt 0.$$
Moreover, if $f: F \lrt F'$ is a morphism in $\mmod \CX$ and $\CX( - ,X'_1) \st{(-,d')} \lrt \CX( - ,X'_0) \lrt F' \lrt 0$ is a projective presentation of $F'$. Then clearly $ f $ can be lifted to a morphism between projective presentations. Yoneda lemma now comes to play for the projective terms to provide unique morphisms $d$, $d'$, $f_0 $ and $ f_1 $ such that the following diagram is commutative
\[\xymatrix{X_1 \ar[r] \ar[d]^{f_1} & X_0 \ar[d]_{f_0}\\ X'_1 \ar[r] & X'_0.}\]
These morphisms then induce a morphism $\va(F) \lrt \va(F')$, which is exactly $\va(f)$. Note that
{ the definition of  $\va(f)$ is independent of choosing a lifting of $ f $} and also as an immediate consequence of \cite[Proposition 3.1]{AHK}, $ \va $ is a covariant functor.


 
$ (ii) $ 
Let $ \CX $ be a contravariantly finite subcategory of $\mmod \Lambda$ containing projective $\Lambda$-modules
and $ M \in  \mmod \Lambda$. Then $\va_{\rho}(M)=\Hom_\Lambda( - ,M)\vert_{\CX}.$

We sometimes write $( -,M)\vert_{\CX}$ for $\Hom_\Lambda( - ,M)\vert_{\CX}$, where it is clear from the context. Also,  if  $M \in \CX$, $\Hom_\Lambda( - ,M)\vert_{\CX}=\CX( - ,M)$.

{ Note that since $ \CX $ is a contravariantly finite subcategory of $\mmod \Lambda$, for every $ \Lambda $-module $ M $, the functor $\Hom_\Lambda( - ,M)\vert_{\CX}$ belongs to $\mmod \CX$.} 
\end{remark}

\begin{remark}\label{Two exact sequences}  
{ $ (i) $ Let $ \CX $ be a contravariantly finite subcategory of $\mmod \Lambda$ containing projective $\Lambda$-modules. 
Then, we have the following  sequence
\[0 \lrt F_0 \lrt F  \lrt  ( - ,\va(F))\vert_{\CX} \lrt F_1 \lrt 0,\]
where $F_0$ and $F_1$ are in $\mmodd\CX$ \cite[Remark 3.10]{AHK}. 

It is worth to note that in case $\CX=\mmod \La$,  this exact sequence is exactly the fundamental exact sequence obtained by Auslander \cite[Page 204, Line 6]{Au1}.}
\vspace{5 mm}

$ (ii) $
Let $F$ and $F'$ be functors in $\mmod\CX.$ By part $ (i) $, there exist exact sequences
\[0 \lrt F_0 \lrt F \st{\alpha} \lrt ( - ,\va(F))\vert_{\CX} \lrt F_1 \lrt 0;\]
\[0 \lrt F'_0 \lrt F' \st{\alpha'} \lrt ( - ,\va(F'))\vert_{\CX} \lrt F'_1 \lrt 0,\]
such that $F_0, F_1, F'_0$ and $F'_1$ are in $\mmodd\CX$. Note that \cite[Lemma 5.1]{AHK} allows us to follow similar argument as in the Proposition 3.4 of \cite{Au1} and deduces that
\[(( - ,\va(F))\vert_{\CX},( - ,\va(F'))\vert_{\CX}) \cong (F, ( - ,\va(F'))\vert_{\CX}).\]
In fact, since $ F_0 $ and $ F_1 $ are in $ \mmodd \CX $, $ \Ext^{1}(F_i, (-,\va(F'))\vert_{\CX})=0 $ for $ i=0 $ and $ 1 $ \cite[Lemma 5.1]{AHK}. Hence we have the exact sequence
\[ (F_1,(-,\va(F'))\vert_{\CX})   \rightarrow (( - ,\va(F))\vert_{\CX}, (-,\va(F'))\vert_{\CX}) \rightarrow (F , (-,\va(F'))\vert_{\CX}) \rightarrow (F_0, (-,\va(F'))\vert_{\CX}).\]
Since $(F_i, (-,\va(F'))\vert_{\CX})=0 $ for $ i=0 $ and $ 1 $ \cite[Lemma 5.1]{AHK}, the map 
$$ { \Sigma:} (( - ,\va(F))\vert_{\CX},( - ,\va(F'))\vert_{\CX}) \longrightarrow (F, ( - ,\va(F'))\vert_{\CX}).$$  is an isomorphism.
Thus given any map $ F \lrt (-,\va(F'))\vert_{\CX} $, there exists one and only one map $ ( - ,\va(F))\vert_{\CX} \lrt (-,\va(F'))\vert_{\CX} $ which makes the diagram
\[\xymatrix{F \ar[r]^{\alpha \ \ \ \ \ \ \ }  \ar@{=}[d]  & ( - ,\va(F))\vert_{\CX} \ar[d]\\ F \ar[r]  & ( - ,\va(F'))\vert_{\CX}}\]
commutative. So for a morphism $\sigma: F \rt F'$ in $\mmod\CX$, there exists a unique map $\delta: ( - ,\va(F))\vert_{\CX} \rt ( - ,\va(F'))\vert_{\CX}$ commuting the following square:
\[\xymatrix{F \ar[r]^{\alpha \ \ \ \ \ \ \ \ }  \ar[d]_{\sigma} & ( - ,\va(F))\vert_{\CX} \ar[d]_{\delta}\\ F' \ar[r]^{\alpha' \ \ \ \ \ \ \ \ }  & ( - ,\va(F'))\vert_{\CX}}\]
Consequently, there are unique morphisms $\sigma_0: F_0 \rt F'_0$ and $\sigma_1: F_1 \rt F'_1$ such that the following diagram is commutative
\[\xymatrix{0 \ar[r] & F_0 \ar@{.>}[d]_{\sigma_0} \ar[r] & F \ar[d]_{\sigma} \ar[r]^{\alpha \ \ \ \ \ \ \ \ } & ( - ,\va(F))\vert_{\CX} \ar[d]^{\delta} \ar[r] & F_1 \ar@{.>}[d]^{\sigma_1} \ar[r] & 0 \\ 0  \ar[r] & F'_0 \ar[r] & F' \ar[r]^{\alpha' \ \ \ \ \ \ \ \ } & (-, \va(F'))\vert_{\CX} \ar[r] & F'_1 \ar[r] & 0.}\]
\end{remark}

\hspace{-5 mm}{\bf Convention}\label{Convention} 
From now to the end, $\CX$ is a contravariantly finite subcategory of $\mmod\La$ containing $\prj\La$.
For a subcategory $\CB$ of an abelian category $\CA$, $\C(\CB)$ [resp. $\K(\CB)$] denotes the category of complexes [resp. the homotopy category of complexes] over $\CB$. Their full subcategories consisting of bounded complexes will be denoted by $\C^{\bb}(\CB)$ and $\K^{\bb}(\CB)$, respectively.

\section{A Derived Version of Auslander's Formula}\label{A Derived Version of Auslander's Formula}
In this section, we give a derived version of { the} Auslander's formula. To this end, we first prove the following proposition.
This proposition has been proved in \cite[Lemma 3.1.9 (i)]{AAHV} in slightly different settings. For the convenient of the reader, we provide a proof with some modifications to { adapt} it with our settings in this section.
Let  $\CX$ be a contravariantly finite subcategory of $\mmod\La$ containing $\prj\La$.
The exact functor $\va:\mmod\CX \lrt \mmod\La$, defined in Remark \ref{Definition of Va} $ (i) $, can be extended to $\D^{\bb}(\mmod\CX)$ to induce a triangle functor $$\D^{\bb}_{\va}:\D^{\bb}(\mmod\CX) \lrt \D^{\bb}(\mmod\La).$$ It acts on objects, as well as roofs, terms by terms (i.e. for every complex $\mathbf{F}=(F^i,\pa^i) $, $ \D^{\bb}_{\va}(\mathbf{F}):=(\va(F^i),\va(\pa^i))$). Let us denote the kernel of $\D^{\bb}_{\va}$ by $\D^{\bb}_{0}(\mmod\CX)$. By definition, it consists of all complexes ${\bf K}$ such that $\D^{\bb}_{\va}({\bf K}) \simeq 0.$ Clearly $\D^{\bb}_{0}(\mmod\CX)$ is a thick subcategory of $\D^{\bb}(\mmod\CX).$ The induced functor
    \[\D^{\bb}(\mmod\CX)/\D^{\bb}_{0}(\mmod\CX) \lrt \D^{\bb}(\mmod\La)\] will be denoted by $\widetilde{\D^{\bb}_{\va}}.$ 

\begin{proposition}\label{ExSeqCom}
Let $\mathbf{F} \in \C(\mmod\CX)$ be a complex over $\mmod\CX$. There exists an exact sequence
\[0 \lrt \mathbf{F}_0 \lrt \mathbf{F} \lrt (-, \D^{\bb}_{\va}(\mathbf{F}))\vert_{\CX} \lrt \mathbf{F}_1 \lrt 0,\]
where $\mathbf{F}_0$ and $\mathbf{F}_1$ are complexes over $\mmodd\CX$.
\end{proposition}

\begin{proof}
Let $\mathbf{F}=(F^i,\pa^i)$. By Remark \ref{Two exact sequences} $ (i) $, for every $i \in \Z$, there is an exact sequence
\[0 \lrt F^i_0 \lrt F^i \st{\alpha^i}  \lrt ( -, \va(F^i))\vert_{\CX} \lrt F^i_1 \lrt 0,\]
such that $F^i_0$ and $ F^i_1$ belong to $\mmodd\CX$.
In view of Remark \ref{Two exact sequences} $ (ii) $, for every $i \in \Z$, there exists a unique morphism $\delta^{i}: \va(F^i) \lrt \va(F^{i+1})$ 
making  the diagram 
\[\xymatrix{ F^i \ar[r] \ar[d]_{\partial^i} & (-, \va(F^i))\vert_{\CX}  \ar[d]^{(-,\delta^i)} \\
 F^{i+1} \ar[r] & (-, \va(F^{i+1}))\vert_{\CX}}\]
commutative. Hence, there exist the unique morphisms $\pa_0^i: F_0^i \lrt F_0^{i+1}$ 
and $\pa_1^i: F_1^i \lrt F_1^{i+1}$ which make the diagram
\[\xymatrix{0 \ar[r] & F_0^i \ar[r] \ar@{.>}[d]^{\pa_0^i} & F^i \ar[r] \ar[d]^{\pa^i} & (-, \va(F^i))\vert_{\CX} \ar[r] \ar[d]^{(-,\delta^i)} & F_1^i \ar[r] \ar@{.>}[d]^{\pa_1^i} & 0 \\
0 \ar[r] & F_0^{i+1} \ar[r] & F^{i+1} \ar[r] & (-, \va(F^{i+1}))\vert_{\CX} \ar[r] & F_1^{i+1} \ar[r] & 0}\]
commutative.
The uniqueness of $\pa^i_0$, $\pa^i_1$ and $\delta^i$ yield the existence of complexes 
$$\mathbf{F}_0 := \cdots \lrt F^{i-1}_0 \st{\partial_0^{i-1}} \lrt F^i_0 \st{\partial^i_0} \lrt F^{i+1}_0 \lrt \cdots,$$
$$\mathbf{F}_1 := \cdots \lrt F^{i-1}_1 \st{\partial_1^{i-1}} \lrt F^i_1 \st{\partial^i_1} \lrt F^{i+1}_1 \lrt \cdots,$$
$$ \mathbf{V} := \cdots \lrt \va(F^{i-1}) \st{\delta^{i-1}} \lrt \va(F^i) \st{\delta^i} \lrt \va(F^{i+1}) \lrt \cdots$$
that fits together to imply the result (Note that, since $ \alpha^{i}(\Lambda) $ is an isomorphism for all $ i \in \mathbb{Z} $, $ \mathbf{V} \simeq \D^{\bb}_{\va}(\mathbf{F})$).
Thus, we get the desired  exact sequence.
\end{proof}

Let $\K^{\bb}_{\La\mbox{-}\ac}(\mmod\CX)$ denote the full subcategory of $\K^b(\mmod\CX)$ consisting of all complexes $\mathbf{F}$ such that
\[\mathbf{F}(\La): \ \cdots \lrt F^{i-1}(\La) \st{\pa^{i-1}_\La} \lrt F^i(\La) \st{\pa^i_\La} \lrt F^{i+1}(\La) \lrt \cdots,\]
is an acyclic complex of abelian groups. If $\mathbf{F}$ is a complex in $\K^{\bb}_{\La\mbox{-}\ac}(\mmod\CX)$, then $\mathbf{F}(P)$ is acyclic, for all $P\in\prj\La$. 
The Verdier quotient $\K^{\bb}(\mmod\CX)/\K^{\bb}_{\La \mbox{-}\ac}(\mmod\CX)$ will be denoted by $\D^{\bb}_{\La}(\mmod\CX)$.
$$\D^{\bb}_{\La}(\mmod\CX):=\K^{\bb}(\mmod\CX)/\K^{\bb}_{\La \mbox{-}\ac}(\mmod\CX).$$

Moreover, we  denote the full subcategory of $\K^{\bb}(\mmod\CX)$ consisting of all acyclic complexes by $\K^{\bb}_{\ac}(\mmod\CX)$. It is a thick triangulated subcategory of $\K^{\bb}_{\La\mbox{-}\ac}(\mmod\CX)$. Consider the Verdier quotient $$\K^{\bb}_{\La\mbox{-}\ac}(\mmod\CX)/\K^{\bb}_{\ac}(\mmod\CX).$$
Clearly, this quotient is a triangulated subcategory of $\D^{\bb}(\mmod\CX)=\K^{\bb}(\mmod\CX)/\K^{\bb}_{\ac}(\mmod\CX)$.

\begin{corollary}\label{qoutient=mmod}
With the assumptions as in our convention, \[\D^{\bb}_{0}(\mmod\CX) \simeq \K^{\bb}_{\La\mbox{-}\ac}(\mmod\CX)/\K^{\bb}_{\ac}(\mmod\CX).\]
\end{corollary}

\begin{proof}
Let $\mathbf{F}$ be a complex in $\D^{\bb}(\mmod\CX)$. For the proof, it is enough to show that if $\D^{\bb}_{\va}(\mathbf{F})$ is an acyclic complex, then $\mathbf{F} \in \K^{\bb}_{\La\mbox{-}\ac}(\mmod\CX).$ But it follows from the exact sequence
\[0 \lrt \mathbf{F}_0 \lrt \mathbf{F} \lrt (-, \D^{\bb}_{\va}(\mathbf{F}))\vert_{\CX} \lrt \mathbf{F}_1 \lrt 0,\]
of the above Proposition. The proof is hence complete.
\end{proof}

Let
\[\Psi: \D^{\bb}_{\La}(\mmod\CX)=\frac{\K^{\bb}(\mmod \CX)}{\K^{\bb}_{\La \mbox{-} \ac}(\mmod \CX)} \lrt \frac{\K^{\bb}(\mmod \CX)/ \K^{\bb}_{\ac}(\mmod \CX)}{\K^{\bb}_{\La \mbox{-} \ac}(\mmod \CX)/ \K^{\bb}_{\ac}(\mmod \CX)}=\frac{\D^{\bb}(\mmod\CX)}{\D^{\bb}_{0}(\mmod\CX)}\]
denote the equivalence of triangulated quotients [V2, Corollaire 4-3]. Clearly $\Psi$ acts as identity on the objects but sends a roof $\frac{f}{s}$ to the roof $\frac{f/1}{s/1}$.

The composition
\[\widetilde{\va}:=\widetilde{\D^{\bb}_{\va}}\Psi:\D^{\bb}_{\La}(\mmod\CX) \lrt \D^{\bb}(\mmod\La)\]
attaches to any complex $\mathbf{F}$ the complex $\widetilde{\va}(\mathbf{F})$, where
\[\widetilde{\va}(\mathbf{F}): \ \cdots \lrt \va(F^{i-1}) \st{\va(\pa^{i-1})} \lrt \va(F^i) \st{\va(\pa^i)} \lrt \va(F^{i+1}) \lrt \cdots.\]
Similarly, $\widetilde{\va}$ sends a roof $\xymatrix{\mathbf{F}  & \ar[l]_s \mathbf{H} \ar[r]^f & \mathbf{G}}$ to the roof
\[\xymatrix{\widetilde{\va}(\mathbf{F}) & \ar[l]_{\widetilde{\va}(s)} \widetilde{\va}(\mathbf{H}) \ar[r]^{\widetilde{\va}(f)} & \widetilde{\va}(\mathbf{F})},\]
where for each morphism $f$ in $\K^{\bb}(\mmod\CX)$, $\widetilde{\va}(f)$ is the homotopy equivalence of a chain map in $\C^{\bb}(\mmod\La)$ obtained by applying $\va$ terms by terms on  the chain map $f$.

\begin{proposition}\label{equi}
The functor $\widetilde{\va}$ is an equivalence of triangulated categories. In particular, the functor $\widetilde{\D^{\bb}_{\va}}$ is an equivalence of triangulated categories.
\end{proposition}

\begin{proof}
Since $\widetilde{\va} =\widetilde{\D^{\bb}_{\va}}\Psi $ and  $ \Psi$ is an equivalence, $\widetilde{\va}$ is an equivalence if and only if so is $\widetilde{\D^{\bb}_{\va}}$. This proves the second part. To prove the first part,
we define the functor  $\eta: \D^{\bb}(\mmod \Lambda) \lrt \D^{\bb}_{\Lambda}(\mmod \CX)$ as follows: 
For every complex $ \mathbf{X},  $ $\eta (\X):=(-,\X)\vert_{\CX}$. Also, $\eta$ maps every roof $\xymatrix{\X  & \ar[l]_s\BZ \ar[r]^f & \Y }$ in $\D^{\bb}(\mmod \Lambda) $ to the roof
$$\xymatrix{ (-,\X)\vert_{\CX} & \ar[l]_{(-,s)} (-,\BZ)\vert_{\CX} \ar[r]^{(-,f)}& (-,\Y)\vert_{\CX} }$$ in $\D^{\bb}_{\Lambda}(\mmod \CX)$. Observe that, since the mapping cone $M(s)$ belongs to $\K^{\bb}_{\ac}(\mmod \Lambda)$, the mapping cone ${M}((-,s))$ belongs to $\K^{\bb}_{\Lambda\mbox{-} \ac}(\mmod \CX)$.
Now, we show that the functor $\eta$ is {\bf faithful}, {\bf full} and {\bf dense}. 
Let $\xymatrix{\X & \ar[r]^f\BZ\ar[l]_s  & \Y }$ be a roof in $\D^{\bb}(\mmod \Lambda)$ such that the induced roof 
$$\xymatrix{(-, \X)\vert_{\CX} & (-, \BZ)\vert_{\CX} \ar[l]_{(-,s)} \ar[r]^{(-,f)}& (-,\Y)\vert_{\CX} }$$ 
is zero in $\D^{\bb}_{\Lambda}(\mmod \CX)$. So, there is a morphism $\xi: {\bf F} \lrt (-,\BZ) $ such that the mapping cone $M(\xi) $ belongs to $ \K^{\bb}_{\Lambda\mbox{-} \ac}(\mmod \CX)$ and $ (-, f) \circ \xi =0$ in $\K^{\bb}(\mmod \CX)$. By Proposition \ref{ExSeqCom},
there exists an exact sequence
$ \xymatrix{ 0 \ar[r] & {\bf F}_0 \ar[r]& {\bf F} \ar[r]^{\alpha \ \ \ \ \ \  }  & (-,  \D^{\bb}_{\va}(\mathbf{F}))\vert_{\CX} \ar[r]  &  {\bf F}_1\ar[r] &0,}$
where $\mathbf{F}_0$ and $\mathbf{F}_1$ are complexes over $\mmodd\CX$.
Since,  $\alpha(\Lambda): {\bf F}(\Lambda) \lrt \D^{\bb}_{\va}(\mathbf{F})$ is an isomorphism, there is a map $\D^{\bb}_{\va}(\mathbf{F})  \st{(\alpha(\Lambda))^{-1}} \lrt {\bf F}(\Lambda) \st{\xi(\Lambda)} \lrt \BZ$ with acyclic cone, such that $f \circ \xi (\Lambda) \circ (\alpha(\Lambda))^{-1}=0$ in $\K^{\bb}(\mmod \Lambda)$.  Hence, the roof  $\xymatrix{\X & \BZ \ar[l]_s\ar[r]^f & \Y }$ is zero in $\D^{\bb}(\mmod \Lambda)$. Thus, $\eta$ is {\bf faithful}.
The same argument as above works to prove that $\eta$ is {\bf full}. Indeed, let $ \xymatrix{(-,\X)\vert_{\CX}  & {\bf H} \ar[r]^f \ar[l]_s & (-, \Y)\vert_{\CX} } $ be a roof in $\D^{\bb}_{\Lambda}(\mmod \CX)$. By Proposition \ref{ExSeqCom}, there exists an exact sequence
$$ 0 \lrt {\bf H}_0 \lrt {\bf H} \st{\beta}\lrt (-,\D^{\bb}_{\va}(\mathbf{H}))\vert_{\CX} \lrt {\bf H}_1 \lrt 0,$$
where ${\bf H}_0$ and ${\bf H}_1$ are complexes over $\mmodd \CX$. 
Proposition \ref{ExSeqCom} allows us to follow similar argument used in 
Remark \ref{Two exact sequences} $ (ii) $ and deduce the isomorphisms
$  ((-, \D^{\bb}_{\va}(\mathbf{H}))\vert_{\CX}, (-, \X)\vert_{\CX}) \cong ({\bf H}, (-, {\bf X})\vert_{\CX})$ and
$  ((-, \D^{\bb}_{\va}(\mathbf{H}))\vert_{\CX}, (-,\Y)\vert_{\CX}) \cong ({\bf H}, (-, \Y)\vert_{\CX}).$
Therefore, there are morphisms $\beta_{\X}: (-, \D^{\bb}_{\va}(\mathbf{H}))\vert_{\CX} \lrt (-, \X)\vert_{\CX}$ and $\beta_{\Y}: (-, \D^{\bb}_{\va}(\mathbf{H}))\vert_{\CX} \lrt (-,\Y)\vert_{\CX}$  such that $\beta_{\X} \circ \beta=s$ and $\beta_{\Y}\circ \beta=f$. Note that since the mapping cone $M (s) \in \K^{\bb}_{\Lambda \mbox{-} \ac}(\mmod \CX)$ and $\beta(\Lambda)$ is an isomorphism,  the mapping cone $M (\beta_{\X})$ belongs to $\K^{\bb}_{\Lambda\mbox{-} \ac}(\mmod \CX)$.
Now,  the commutative diagram
 \[\xymatrix@C-0.8pc@R-0.5pc{ & & {\bf H} \ar[dr]^{\beta} \ar[dl]_{\rm id}& & \\
 & {\bf H} \ar[dl]_s \ar[drrr]^{ f  \ \ } && (-, \D^{\bb}_{\va}(\mathbf{H}))\vert_{\CX}\ar[dr]^{\beta_{\Y}} \ar[dlll]_{\beta_{\X}} &
 \\ (-, \X)\vert_{\CX}  & & & & (-,\Y)\vert_{\CX}   }\]
 implies that the roof $\xymatrix{ (-,\X)\vert_{\CX}  & {\bf H} \ar[r]^f \ar[l]_{\ \ \ s } & (-, \Y)\vert_{\CX} }$ is equivalent to the roof $$\xymatrix{(-,\X)\vert_{\CX}  & (-, \D^{\bb}_{\va}(\mathbf{H}))\vert_{\CX}\ar[r]^{\beta_{\X}} \ar[l]_{  \beta_{\Y}}& (-, \Y)\vert_{\CX} } $$ in $\D^{\bb}_\Lambda(\mmod \CX)$.

 Moreover, let $\X$ be a complex in $\D^{\bb}_{\Lambda}(\mmod \CX)$. Then an exact sequence
$$0 \lrt \X_0 \lrt \X \lrt (-, \D^{\bb}_{\va}(\mathbf{X}))\vert_{\CX} \lrt \X_1 \lrt 0$$
implies that $\X$ is isomorphic to $(-, \D^{\bb}_{\va}(\mathbf{X}))\vert_{\CX}$ in $\D^{\bb}_{\Lambda}(\mmod \CX)$  and so $\eta$ is {\bf dense}.

Obviously, $ \widetilde{\va}\circ\eta \simeq 1_{\D^{\bb}(\mmod \Lambda)} $ and so $ \widetilde{\va} $ is an equivalence.
\end{proof}

\begin{remark}
The triangle-equivalence
$$\widetilde{\D^{\bb}_{\va}}: \D^{\bb}(\mmod \CX)/ \D^{\bb}_{0}(\mmod \CX) \lrt \D^{\bb}(\mmod \La).$$
is in fact a derived version of  the Auslander{\bf ' s} formula. This derived level formula has been proved by Krause { \cite[Corollary 3.2]{Kr1}} for the case where $\CX=\mmod\La$.
\end{remark}

\section{Categorical resolutions of bounded derived categories}\label{Last Section}
In this section, we give a functorial approach to show that $\D^b(\mmod\La)$, the bounded derived category of $\La$, admits a categorical resolution, where $\La$ is an arbitrary Artin algebra.

We begin by the definition of a categorical resolution of the bounded derived category of an Artin algebra. Although, the definition in literature is for arbitrary triangulated categories, in this paper we only concentrate on the bounded derived categories of Artin algebras. 

We follow the definition presented by \cite[Definition 2.2]{Z}, which is a combination of a definition due to Bondal and Orlov \cite{BO} and also another one due to Kuznetsov \cite[Definition 3.2]{Kuz}, both as different attempts for providing a categorical translation of the notion of the resolutions of singularities.

\begin{definition} (\cite[Definition 2.2]{Z})\label{Categorical Resolution}
Let $\La$ be an Artin algebra of infinite global dimension. A categorical resolution of $\D^{\bb}(\mmod\La)$, is a triple $(\D^{\bb}(\mmod\La'), \pi_{*}, \pi^{*})$, where $\La'$ is an Artin algebra of finite global dimension and $\pi_*: \D^{\bb}(\mmod\La') \lrt \D^{\bb}(\mmod\La)$ and $\pi^*: \K^{\bb}(\prj\La) \lrt \D^{\bb}(\mmod\La')$ are triangle functors satisfying the following conditions:
\begin{itemize}
\item[$(i)$] $\pi_*$ induces a triangle-equivalence $\frac{\D^{\bb}(\mmod \La')}{\Ker(\pi_*)} \simeq \D^{\bb}(\mmod \La)$;
\item[$(ii)$] $\pi^*$ is left adjoint to $\pi_*$ on $\K^{\bb}(\prj\La)$. That is, for every ${\bf P} \in \K^{\bb}(\prj\La)$ and every ${\bf X'} \in \D^{\bb}(\mmod\La')$, there exists a functorial isomorphism \[\D^{\bb}(\mmod\La')(\pi^*({\bf P}), {\bf X'}) \cong \D^{\bb}(\mmod\La)({\bf P}, \pi_*({\bf X'}));\]
\item[$(iii)$] The unit $\eta: 1_{\K^b(\prj\La)} \lrt \pi_*\pi^*$ is a natural isomorphism.
\end{itemize}
Furthermore, a categorical resolution $(\D^{\bb}(\mmod\La'), \pi_{*}, \pi^{*})$ of $\D^{\bb}(\mmod\La)$ is called { \it weakly crepant} if $\pi^*$ is also a right adjoint to $\pi_*$ on $\K^b(\prj\La)$.
\end{definition}

Now we are in a position to give a functorial approach to show that the bounded derived category of every Artin algebra admits a categorical resolution.

\begin{theorem}\label{main4}
Let $\La$ be an Artin algebra of infinite global dimension and $\widetilde{\La}$ denote its A-algebra. Then $(\D^{\bb}(\mmod\widetilde{\La}), \D^{\bb}_{\va}, \K^{\bb}_{\va_{\la}})$ is a categorical resolution of $\D^{\bb}(\mmod\La)$.
\end{theorem}

\begin{proof}
Let $\CX$ be a contravariantly finite subcategory of $\mmod \Lambda$ containing $\prj \Lambda$. We prove the theorem in 6 steps.

{\bf Step I.} As it is mentioned in Section \ref{A Derived Version of Auslander's Formula}, the exact functor $\va:\mmod\CX \lrt \mmod\La$, defined in Remark \ref{Definition of Va} $ (i) $, can be extended naturally to $\D^{\bb}(\mmod\CX)$ to induce a triangle functor $$\D^{\bb}_{\va}:\D^{\bb}(\mmod\CX) \lrt \D^{\bb}(\mmod\La).$$ It acts on objects, as well as roofs, terms by terms.

{\bf Step II.} 
{ We define the functor $\K^{\bb}_{\va_{\la}}: \K^{\bb}(\prj\La) \lrt \K^{\bb}(\mmod\CX)$ as follows: For every complex $\mathbf{P} \in \K^{\bb}(\prj\La)$, $\K^{\bb}_{\va_{\la}}(\mathbf{P}):=( - , \mathbf{P})$. So in fact, it is a functor from $\K^{\bb}(\prj\La)$ to $\K^{\bb}(\prj(\mmod\CX))$. That is, for every complex $\mathbf{P} \in \K^{\bb}(\prj\La)$, $\K^{\bb}_{\va_{\la}}(\mathbf{P})=( - , \mathbf{P})$ is a bounded complex of projective $\CX$-modules.}

{\bf Step III.} As it is mentioned in Section \ref{A Derived Version of Auslander's Formula},
denote the kernel of $\D^{\bb}_{\va}$ by $\D^{\bb}_{0}(\mmod\CX)$. 
By Proposition \ref{equi}, the induced functor
\[\frac{\D^{\bb}(\mmod\CX)}{\D^{\bb}_{0}(\mmod\CX)} \lrt \D^{\bb}(\La),\] 
denoted by $\widetilde{\D^{\bb}_{\va}},$ is an equivalence of triangulated categories. 

{\bf Step IV.} $\K^{\bb}_{\va_{\la}}$ is left adjoint to $\D^{\bb}_{\va}$ on $\K^{\bb}(\prj\La)$. To prove this we show that for every complexes $\mathbf{P} \in \K^{\bb}(\prj\La)$ and $\mathbf{F} \in \D^{\bb}(\mmod\CX)$, there exists an isomorphism
\[\D^{\bb}(\mmod\CX)(\K^{\bb}_{\va_{\la}}(\mathbf{P}), \mathbf{F}) \cong \D^{\bb}(\mmod\La)(\mathbf{P}, \D^{\bb}_{\va}(\mathbf{F})),\]
of abelian groups.
Let $\mathbf{F} \in \D^{\bb}(\mmod\CX)$. 
By Proposition \ref{ExSeqCom}, there exists an exact sequence of complexes
$0 \lrt \mathbf{F}_0 \lrt \mathbf{F} \lrt (-, \D^{\bb}_{\va}(\mathbf{F}))\vert_{\CX} \lrt \mathbf{F}_1 \lrt 0,$
such that $\mathbf{F}_0$ and $\mathbf{F}_1$ are complexes over $\mmodd\CX$. This sequence can be divided into the following two short exact sequences of complexes
\[0 \lrt \mathbf{F}_0 \lrt \mathbf{F}\lrt \mathbf{K} \lrt 0 \ \ \ {\rm and} \ \ \ 0 \rt  \mathbf{K} \lrt ( - , \D^{\bb}_{\va}(\mathbf{F})) \lrt \mathbf{F}_1 \lrt 0.\]
These two sequences, in turn, induce the following two triangles
$ \mathbf{F}_0 \lrt \mathbf{F} \lrt \mathbf{K} \rightsquigarrow $ and $ \mathbf{K} \lrt ( - , \D^{\bb}_{\va}(\mathbf{F})) \lrt \mathbf{F}_1 \rightsquigarrow,$
in $\D^{\bb}(\mmod\CX)$, where $\mathbf{F}_0$ and $\mathbf{F}_1$ are considered as objects of $\D^{\bb}(\mmodd\CX)$.
Applying the functor $\D^{\bb}(\mmod \CX)(\K^{\bb}_{\va_{\la}}(\mathbf{P}), - )$ on these triangles, there are the induced two long exact sequences of abelian groups
\[(\K^{\bb}_{\va_{\la}}(\mathbf{P}), \mathbf{F_0}) \lrt (\K^{\bb}_{\va_{\la}}(\mathbf{P}), \mathbf{F}) \lrt (\K^{\bb}_{\va_{\la}}(\mathbf{P}), \mathbf{K}) \lrt (\K^{\bb}_{\va_{\la}}(\mathbf{P}), \mathbf{F_0}[1])\] and
\[(\K^{\bb}_{\va_{\la}}(\mathbf{P}), \mathbf{F_1}[-1]) \lrt (\K^{\bb}_{\va_{\la}}(\mathbf{P}), \mathbf{K}) \lrt (\K^{\bb}_{\va_{\la}}(\mathbf{P}), ( - , \D^{\bb}_{\va}(\mathbf{F}))) \lrt (\K^{\bb}_{\va_{\la}}(\mathbf{P}), \mathbf{F_1}),\]
respectively, where all Hom groups are taken in $\D^{\bb}(\mmod\CX)$. 
\\
But since $\mathbf{P} \in \K^{\bb}(\prj \La)$, $\K^{\bb}_{\va_{\la}}(\mathbf{P})=(-,\mathbf{P}) \in \K^{\bb}(\prj(\mmod\CX))$. Hence,  by applying some known abstract facts in triangulated categories, e.g. \cite[Corollary 10.4.7]{W},  all these Hom sets can be also considered in $\K^{\bb}(\mmod\CX)$.
 
{ On the other hand by applying Yoneda lemma  terms by terms, one can easily show  that for every $\mathbf{P} \in \K^{\bb}(\prj\La)$ and $\mathbf{G} \in \K^{\bb}(\mmodd\CX)$,
$ \K^{\bb}(\mmod\CX)(( - ,\mathbf{P}), \mathbf{G})=0$. Thus, there exists the following isomorphism of abelian groups.}
$$\K^{\bb}(\mmod\CX)(\K^{\bb}_{\va_{\la}}(\mathbf{P}), \mathbf{F}) \cong \K^{\bb}(\mmod\CX)(\K^{\bb}_{\va_{\la}}(\mathbf{P}), ( - , \D^{\bb}_{\va}(\mathbf{F}))).$$
Therefore, to complete the proof, it is enough to show that
$$\K^{\bb}(\mmod\CX)(\K^{\bb}_{\va_{\la}}(\mathbf{P}), ( - , \D^{\bb}_{\va}(\mathbf{F}))) \cong  \K^{\bb}(\mmod\La)(\mathbf{P}, \D^{\bb}_{\va}(\mathbf{F})).$$
This is a consequence of Yoneda lemma applying terms by terms in view of the fact that $\K^{\bb}_{\va_{\la}}(\mathbf{P})=( - , \mathbf{P})$. 
Note that since $\mathbf{P}$ is a bounded complex of projectives, by \cite[Corollary 10.4.7]{W}, { $\D^{\bb}(\mmod \Lambda) (\mathbf{P}, \D^{\bb}_{\va}(\mathbf{F})) \cong \K^{\bb}(\mmod \Lambda) (\mathbf{P}, \D^{\bb}_{\va}(\mathbf{F}))$}. The proof of this Step is now complete.

\vspace{10 mm}

{\bf Step V.} For every bounded complex of projectives $\mathbf{P}$, $ \D^{\bb}_{\va}\K^{\bb}_{\va_{\la}}(\mathbf{P})=\D^{\bb}_{\va}(( - ,\mathbf{P}))=\mathbf{P}.$

{\bf Step VI.} Since $\widetilde{{\La}}$ is A-algebra of $ \Lambda $, there is a $ \Lambda $-module $ M $ such that $\widetilde{{\La}}=\End_{\La}(M)$ { \cite[Page 47, Line 26]{Au2}}. Set $\CX:=\add M$. Then $\mmod\CX \simeq \mmod \widetilde{\La}.$ Therefore, the triple
{\bf $(\D^{\bb}(\widetilde{\La}), \D^{\bb}_{\va}, \K^{\bb}_{\va_{\la}})$} is a categorical resolution of $\D^{\bb}(\La)$.
\end{proof}

Towards the end of the paper, we show that if $\La$ is a self-injective Artin algebra of infinite global dimension, then the triple $(\D^{\bb}(\mmod\widetilde{\La}), \D^{\bb}_{\va}, \K^{\bb}_{\va_{\la}})$ introduced in the above theorem, provides a { \it weakly crepant} categorical resolution of $\D^{\bb}(\mmod\La)$. To do this, we need some lemmas.

\begin{lemma}\label{injective object}
Let $\La$ be an Artin algebra,  $\CX$ be a contravariantly finite subcategory of $\mmod \Lambda$ containing $\prj \Lambda$, and $I \in \inj\La$. Then the functor $(-,I)\vert_{\CX}$ is an injective object of $\mmod\CX.$
\end{lemma}

\begin{proof}
Since $\CX$ is a contravariantly finite subcategory of $\mmod\La$, there exists an exact sequence $X_1 \st{d_1}\lrt X_0 \st{d_0}\lrt I \lrt 0$ of $\La$-modules such that $d_0$ and $d_1$ are right $\CX$-approximations of $I$ and $\Ker (d_0)$, respectively. This guarantees the existence of the exact sequence
\[(-,X_1) \st{(-,d_1)}\lrt (-,X_0) \st {(-,d_0)}\lrt (-,I)\vert_{\CX}\lrt 0\]
in $\mmod\CX.$ Hence $(-,I)\vert_{\CX}$ is a finitely presented functor. To show that it is injective, pick a short exact sequence $0 \rt F' \rt F \rt F'' \rt 0$ of $\CX$-modules and apply the functor $( - ,(-,I)\vert_{\CX})$ on it to get the sequence
$ 0 \lrt (F'', (-,I)\vert_{\CX}) \lrt (F,(-,I)\vert_{\CX}) \lrt (F',(-,I)\vert_{\CX}) \lrt 0.$
As an immediate consequence of \cite[Theorem 3.7]{AHK}, $(\va,\va_{\rho})$ is an adjoint pair and so we have the following commutative diagram
\[\xymatrix{0 \ar[r] & (F'', (-,I)\vert_{\CX})  \ar[r] \ar[d] & (F,(-,I)\vert_{\CX}) \ar[r] \ar[d] &(F',(-,I)\vert_{\CX})   \ar[d] \ar[r] & 0 \\ 0 \ar[r] & (\va(F''), I) \ar[r] &  (\va(F), I) \ar[r] & (\va(F'), I)  \ar[r]  & 0,}\]
where the vertical arrows are isomorphisms. But, the lower row is exact, because $I$ is an injective module and $\va$ is an exact functor by Remark \ref{Definition of Va}. Hence the upper row should be exact, that implies the result.
\end{proof}

\begin{remark}\label{HomonDorK}
Let $\La$ be a self-injective Artin algebra. So $\prj\La=\inj\La$. Hence a complex $\mathbf{P} \in \K^{\bb}(\prj\La)$ is also a bounded complex of injectives. So by the above lemma, $\K^{\bb}_{\va_{\la}}(\mathbf{P})=( - ,\mathbf{P})$ is a complex of injective $\CX$-modules. Therefore, by \cite[Corollary 10.4.7]{W}, all Hom sets with either $\mathbf{P}$ or $\K^{\bb}_{\va_{\la}}(\mathbf{P})$ in the second variants, can be calculated either in $\D^{\bb}(\mmod \Lambda)$ or in $\D^{\bb}(\mmod\CX)$, i.e. for every complex $ \mathbf{F} $ in $ \K^{\bb}(\mmod\CX) $,
$$\K^{\bb}(\mmod\CX)(\mathbf{F}, \K^{\bb}_{\va_{\la}}(\mathbf{P})) \cong \D^{\bb}(\mmod \CX)(\mathbf{F}, \K^{\bb}_{\va_{\la}}(\mathbf{P})).$$
\end{remark}

\begin{lemma}\label{HomInj=0}
Let $\La$ and $\CX$ be as in our convention. Then for every complexes $\mathbf{G} \in \K^{\bb}(\mmodd\CX)$ and $\mathbf{M} \in \K^{\bb}(\mmod\La)$,
\[\K^{\bb}(\mmod\CX)(\mathbf{G}, ( - ,\mathbf{M})\vert_{\CX})=0.\]
\end{lemma}

\begin{proof}
Let $\mathbf{G}=(G^i,\pa^i_{\mathbf{F}})$ and $\mathbf{M}=(M^i,\pa^i_{\mathbf{M}})$. As an immediate consequence of \cite[Theorem 3.7]{AHK}, $(\va,\va_{\rho})$ is an adjoint pair and so for every $i \in \Z$, we have the isomorphism
$$\mmod\CX({G^i},( - ,{M^i})\vert_{\CX})\cong \mmod\La(\va({G^i}),{M^i}).$$
Hence $\mmod\CX({G^i},( - ,{M^i})\vert_{\CX})=0$, because ${G^i} \in \mmodd\CX=\Ker(\va)$. This fact can be extended naturally, terms by terms, to bounded complexes to prove the lemma.
\end{proof}

\begin{theorem}\label{main5}
Let $\La$ be a self-injective Artin algebra of infinite global dimension. Then the triple $(\D^{\bb}(\mmod\widetilde{\La}), \D^{\bb}_{\va}, \K^{\bb}_{\va_{\la}})$ introduced in Theorem \ref{main4}, is a { \it weakly crepant} categorical resolution of $\D^{\bb}(\mmod\La)$.
\end{theorem}

\begin{proof}
We just should show that $\K^{\bb}_{\va_{\la}}$ is a right adjoint of $\D^{\bb}_{\va}$ on $\K^{\bb}(\prj\La)$. Pick $\mathbf{F} \in \D^{\bb}(\mmod\CX)$ and $\mathbf{P} \in \K^{\bb}(\prj\La)$. As it is mentioned in {\bf Step IV}, there exist the two triangles
$\mathbf{F}_0 \lrt \mathbf{F} \lrt \mathbf{K} \rightsquigarrow$ and $  \mathbf{K} \lrt ( - , \D^{\bb}_{\va}(\mathbf{F})) \lrt \mathbf{F}_1 \rightsquigarrow,$
where $\mathbf{F}_0$ and $\mathbf{F}_1$ are objects of $\D^{\bb}(\mmodd\CX)$. By applying the functor $\D^{\bb}(\mmod\CX)( - ,\K^{\bb}_{\va_{\la}}(\mathbf{P}))$ on these triangles, we get two exact sequences of abelian groups. Now we use Remark \ref{HomonDorK} and Lemma \ref{HomInj=0}, to conclude the isomorphism
$$\K^{\bb}(\mmod\CX)(\mathbf{F}, \K^{\bb}_{\va_{\la}}(\mathbf{P})) \cong \K^{\bb}(\mmod\CX)(( - ,\D^{\bb}_{\va}(\mathbf{F})), \K^{\bb}_{\va_{\la}}(\mathbf{P})).$$
The extended version of Yoneda lemma finally helps us to establish the isomorphism
\[\K^{\bb}(\mmod\CX)(( - ,\D^{\bb}_{\va}(\mathbf{F})), \K^{\bb}_{\va_{\la}}(\mathbf{P})) \cong \K^{\bb}(\mmod\La)(\D^{\bb}_{\va}(\mathbf{F}), \mathbf{P})\]
of abelian groups. The proof is hence complete.
\end{proof}

\section*{Acknowledgments}
The authors would like to thank the referees for carefully reading the manuscript and many useful comment and hints that improved our exposition.

\end{document}